\documentclass[11pt]{article}
\usepackage{amsthm,amssymb,amsmath}
\usepackage{epsfig}
\usepackage{verbatim}

\title{Two results on the digraph chromatic number}

\author{Ararat Harutyunyan\thanks{Research supported by FQRNT
  (Le Fonds qu\'{e}b\'{e}cois de la recherche sur la nature et les technologies)
  doctoral scholarship.}\\
  {Department of Mathematics}\\
  {Simon Fraser University}\\
  {Burnaby, B.C. V5A 1S6} \\
  email: {\tt aha43@sfu.ca}
\and
  Bojan Mohar\thanks{Supported in part by an NSERC Discovery Grant (Canada),
  by the Canada Research Chair program, and by the
  Research Grant P1--0297 of ARRS (Slovenia).}~\thanks{On leave from:
  IMFM \& FMF, Department of Mathematics, University of Ljubljana, Ljubljana,
  Slovenia.}\\
  {Department of Mathematics}\\
  {Simon Fraser University}\\
  {Burnaby, B.C. V5A 1S6} \\
  email: {\tt mohar@sfu.ca}
}

\newtheorem{theorem}{Theorem}[section]

\newtheorem{conjecture}[theorem]{Conjecture}
\newtheorem{claim}{Claim}

\newcommand{\DEF}[1]{{\em #1\/}}

\newcommand{\PP}{\mathbb P}
\newcommand{\EE}{\mathbb E}
\newcommand{\tDelta}{\tilde{\Delta}}

\begin{document}

\maketitle

\begin{abstract}

It is known (Bollob\'{a}s \cite{B1978}; Kostochka and Mazurova
\cite{KM1977}) that there exist graphs of maximum degree $\Delta$
and of arbitrarily large girth whose chromatic number is at least
$c \Delta / \log \Delta$. We show an analogous result for digraphs
where the chromatic number of a digraph $D$ is defined as the
minimum integer $k$ so that $V(D)$ can be partitioned into $k$
acyclic sets, and the girth is the length of the shortest cycle in
the corresponding undirected graph. It is also shown, in the same
vein as an old result of Erd\H{o}s \cite{E1962}, that there are
digraphs with arbitrarily large chromatic number where every large
subset of vertices is 2-colorable.

\end{abstract}


\section{Digraph Colorings}

Let $D$ be a (loopless) digraph. A vertex set $A \subset V(D)$ is
called \DEF{acyclic} if the induced subdigraph $D[A]$ has no
directed cycles. A \DEF{$k$-coloring} of $D$ is a partition of
$V(D)$ into $k$ or fewer acyclic sets. The minimum integer $k$ for which
there exists a $k$-coloring of $D$ is the \DEF{chromatic number}
$\chi(D)$ of the digraph D. This definition of the chromatic
number of a digraph was first treated by Neumann-Lara
\cite{N1982}. The same notion was independently introduced two
decades later when considering the circular chromatic number of
weighted (directed or undirected) graphs \cite{M2003}, and further treated
in \cite{BFJKM2004}.

This notion of colorings of digraphs turns out to be the
natural way of extending the theory of undirected  graph colorings
since it provides extensions of most of the basic results from
graph coloring theory \cite{BFJKM2004, HM2011, HM2010, M2003, M2010}.

In this note we prove, using standard probabilistic approach, that
two further analogues of graph coloring results carry over to
digraphs. The first result, see Theorem \ref{thm:1}, provides evidence that the digraph
chromatic number, like the graph chromatic number, is a global
parameter that cannot be deduced from local considerations. The second
result, see Theorem \ref{thm:2}, shows that there are digraphs with large
chromatic number $k$ in which every set of at most $c|V(D)|$ vertices is
2-colorable, where $c>0$ is a constant that only depends on $k$. The 
analogous result for digraphs was proved by Erd\H{o}s \cite{E1962}
with its outcome being that all sets of at most $cn$ are 3-colorable. 
Both the 3-colorability in Erd\H{o}s' result and 2-colorability in Theorem
\ref{thm:2} are best possible.

Concerning the first result, it is well-known that there exist graphs 
with large girth and
large chromatic number. Bollob\'{a}s \cite{B1978} and, independently,
Kostochka and Mazurova
\cite{KM1977} proved that there exist graphs of maximum degree
at most $\Delta$
and of arbitrarily large girth whose chromatic number is
$ \Omega( \Delta / \log \Delta)$. Our Theorem \ref{thm:1} provides
an extension to digraphs.

The bound of $ \Omega( \Delta / \log \Delta)$ from \cite{B1978, KM1977} is
essentially best
possible: a result of Johansson \cite{J1996} shows that if $G$ is
triangle-free, then the chromatic number is $O (\Delta / \log
\Delta)$. Similarly, Theorem \ref{thm:2} is also essentially best possible:
Erd\H{o}s et al. \cite{EGK1991} showed that
every tournament on $n$ vertices has
chromatic number $O(\tfrac{n}{\log n})$.

In general, it may be true that the following analog of Johansson's result
holds for digon-free digraphs, as conjectured by McDiarmid and
Mohar \cite{MM2002}.

\begin{conjecture} \label{conj:1}
Every digraph $D$ without digons and with
maximum total degree $\Delta$ has $\chi(D) =
O(\frac{\Delta}{\log \Delta})$.
\end{conjecture}

Theorem \ref{thm:1} shows that Conjecture \ref{conj:1}, if true,
is essentially best possible.

\section{Chromatic number and girth}
First, we need some basic definitions. For an extensive treatment of
digraphs, we refer the reader to \cite{BG2001}.
Given a loopless digraph $D$, a \DEF{cycle} in $D$
is a cycle in the underlying undirected graph. The \DEF{girth} of
$D$ is the length of a shortest cycle in $D$, and the \DEF{digirth}
of $D$ is the length of a shortest directed cycle in $D$. The \DEF{total
degree} of a vertex $v$ is the number of arcs incident to $v$. The
\DEF{maximum total degree} of $D$, denoted by $\Delta(D)$, is
the maximum of all total degrees of vertices in $D$. The \DEF{out-degree} 
and the \DEF{in-degree} of a vertex $v$ are denoted by $d^{+}(v)$ and $d^{-}(v)$,
respectively.

It is proved in \cite{BFJKM2004} that there are digraphs of
arbitrarily large digirth and dichromatic number. Our result is an
analogue of the aforementioned result of Bollob\'{a}s \cite{B1978} and Kostochka and Mazurova
\cite{KM1977}. Note that the
result involves the girth and not the digirth.

\begin{theorem} \label{thm:1}
Let $g$ and $\Delta$ be positive integers. There exists a digraph $D$ of girth at least $g$,
with $\Delta(D) \leq \Delta$, and $\chi(D) \geq a \Delta / \log \Delta$
for some absolute constant $a > 0$. For $\Delta$ sufficiently large we may take $a= \frac{1}{15}$.
\end{theorem}

\begin{proof}
Our proof is in the spirit of Bollob\'{a}s \cite{B1978}. We may assume that 
$\Delta$ is sufficiently large. 

Let $D=D(n,p)$ be a random digraph of order $n$ defined as follows. 
For every $u,v \in V(D)$, we connect $uv$ with probability
$2p$, independently. Now we randomly (with probability 1/2) assign
an orientation to every edge that is present. Observe that $D$ has no
digons. We will use the value $p = \frac{\Delta}{4en}$, where
$e$ is the base of the natural logarithm.

\begin{claim}
$D$ has no more than
$\Delta^g$ cycles of length less than $g$ with probability at least
$1 - \tfrac{1}{\Delta}$.
\end{claim}

\begin{proof}
Let $N_l$ be the number of cycles of length $l$ in $D$. Then $$\EE[N_l] \leq
\binom{n}{l}l! (2p)^l \leq n^l (2p)^l \leq (\tfrac{\Delta}{4})^l.$$ Therefore, the
expected number of cycles of length less than $g$ is at most
$\Delta^{g-1}$. So the probability that $D$ has more than
$\Delta^{g}$ cycles of length less than $g$ is at most $1/\Delta$
by Markov's inequality.
\end{proof}

\begin{claim} There is a set $A$ of at most $n/1000$ vertices
of $D$ such that $\Delta(D - A) \leq \Delta$ with probability
at least $\tfrac{1}{2}$.

\end{claim}

\begin{proof}
Let $X_d$ be the number of
vertices of total degree $d$, $d=0,1,..., n-1$. Following \cite{B1978}, 
define the \emph{excess degree of} $D$ to be $ex(D) = \sum_{d= \Delta +
1}^{n-1}(d-\Delta)X_d$. 
Clearly, there is a set of at most $ex(D)$ arcs (or vertices) whose removal reduces
the maximum total degree of $D$ to at most $\Delta$. 

Now, we estimate the expectation of $X_d$. By linearity of expectation, we have:

\begin{eqnarray*}
\EE[X_d] &\leq  & n \binom{n-1}{d}(2p)^{d} \\
& \leq & n \left ( \frac{e(n-1)}{d} \right)^d \left ( \frac{\Delta}{2 e n} \right)^d \\
& \leq & n \left( \frac{\Delta}{2d} \right)^d. \\
\end{eqnarray*}

Therefore, by linearity of expectation we have that

\begin{eqnarray*}
\EE[ex(D)] &\leq  & \sum_{d = \Delta + 1}^{n-1} nd \left( \frac{\Delta}{2d} \right)^d \\
& \leq & \frac{n \Delta}{2} \sum_{d = \Delta + 1}^{n-1} \left( \frac{\Delta}{2d} \right)^{d-1}\\
& \leq & \frac{n \Delta}{2} \sum_{d = \Delta + 1}^{n-1} \left( \frac{1}{2} \right)^{d-1}\\
& \leq & \frac{n \Delta}{2} \cdot \frac{(\frac{1}{2})^{\Delta}}{1-\frac{1}{2}}\\
& = & n \cdot \frac{\Delta}{2^{\Delta}}\\
& \leq & \frac{n}{2000}.\\
\end{eqnarray*}

Now, by Markov's inequality, $\PP[ex(D) > n / 1000] < 1/2.$
\end{proof}

Let $\alpha(D)$ be the size of a maximum acyclic set of
vertices in $D$. The following result will be used in the proof
of our next claim and also in Section~3.

\begin{theorem}[\cite{SS2008}] \label{thm:mas}
Let $D \in D(n,p)$. There is an absolute constant $W$ such that if $p$ satisfies $np \geq W$, then,
a.a.s. $$ \alpha(D) \leq \left ( \frac{2}{\log q} \right ) (\log np + 3e), $$
where $q = (1-p)^{-1}$.
\end{theorem}

\begin{claim} Let $D \in D(n,p)$. Then $\alpha(D) \leq \frac{4en \log
\Delta}{\Delta}$ with high probability.
\end{claim}

\begin{proof}
Since $\Delta$ is sufficiently large, Theorem \ref{thm:mas} applies
and the result follows.
\end{proof}

Now, pick a digraph $D$ that satisfies claims 1, 2 and 3. After
removing at most $ n/1000 + \Delta^g \leq n / 100$ vertices, the resulting digraph
$D^{*}$ has maximum degree at most $\Delta$ and girth at least
$g$. Clearly, $\alpha(D^{*}) \leq \alpha(D)$. Therefore,
$\chi(D^{*}) \geq \frac{n(1- 1/100)}{ 4 en \log \Delta / \Delta}
\geq \frac{\Delta}{5e \log \Delta}$.
\end{proof}

\section{Another result of the same nature}

A result of Erd\H{o}s \cite{E1962} states that there exist graphs
of large chromatic number where every induced subgraph with up to a constant
fraction number of the vertices is 3-colorable. In particular, it
is proved that for every $k$ there exists $\epsilon > 0$ such that
for all $n$ sufficiently large there exists a graph $G$ of order
$n$ with $\chi(G)
> k$ and yet $\chi(G[S]) \leq 3$ for every $S \subset V(G)$ with $|S|
\leq \epsilon n$.

The 3-colorability in the aforementioned theorem cannot be improved.
A result of Kierstead, Szemeredi and Trotter \cite{KST1984} (with later
improvements by Nilli \cite{N1999} and Jiang \cite{J2001}) shows that every $4$-chromatic graph of
order $n$ contains an odd cycle of length at most $8 \sqrt{n}$.

We prove the following analog for digraphs. Our proof follows
the ideas of Erd\H{o}s found in \cite{AS1992}.

\begin{theorem} \label{thm:2}

For every $k$, there exists $\epsilon > 0$ such that for every
sufficiently large integer $n$ there exists a digraph $D$ of order $n$ with
$\chi(D)
> k$ and yet $\chi(D[S]) \leq 2$ for every $S \subset V(D)$ with $|S|
\leq \epsilon n$.
\end{theorem}

\begin{proof}
Clearly, we may assume that $\log k \geq 3$ and $k \geq \sqrt{W}$,
where $W$ is the constant in Theorem \ref{thm:mas}. Let us consider
the random digraph $D=D(n,p)$ with $p = \frac{k^2}{n}$ and let
$ 0< \epsilon < k^{-5}$.

We first show that $\chi(D) > k$ with high probability. Since $k$ is sufficiently large,
Theorem \ref{thm:mas} implies that $\alpha(D) \leq 6n \log k / k^2$ with high probability.
Therefore, almost surely $\chi(D) \geq \tfrac{1}{6} k^2 / \log k > k.$

Now, we show that with high probability every set of at most $\epsilon n$ vertices can
be colored with at most two colors. Suppose there
exists a set $S$ with $|S| \leq \epsilon n$ such that $\chi(D[S]) \geq 3$.
Let $T \subset S$ be a 3-critical subset, i.e. for every
$v \in T$, $\chi(D[T]-v) \leq 2$. Let $t = |T|$. Since $D[T]$ is $3$-critical, every 
$v \in T$ satisfies $\min\{d^{+}_{D[T]}(v), d^{-}_{D[T]}(v)\} \geq 2$ for otherwise
a 2-coloring of $D[T]-v$ could be extended to $D[T]$. This implies
that $D[T]$ has at least $2t$ arcs. The probability of this event is at most

\begin{eqnarray}\sum_{3 \leq t \leq \epsilon n} \binom{n}{t} \binom{2 \binom{t}{2} }{2t} \left
(\frac{k^2}{n} \right)^{2t} & \leq & \sum_{3 \leq t \leq \epsilon n}
\left( \frac{en}{t} \right)^t \left( \frac{et(t-1)}{2t} \right)^{2t}
\left(\frac{k^2}{n} \right)^{2t} \notag \\
&\leq & \sum_{3 \leq t \leq \epsilon n} \left( \frac{e^3 tk^4}{4n} \right)^t \notag \\
&\leq & \epsilon n \max_{3 \leq t \leq \epsilon n} \left(
\frac{7tk^4}{n} \right)^t
\end{eqnarray}

If $3 \leq t \leq (\log n)^2$, then $(7tk^4 /n)^t \leq (7 (\log n)^2 k^4 / n)^t \leq
\left (7 (\log n)^2 k^4 / n \right)^3 = o(1/n).$
Similarly, if $(\log n)^2 \leq t
\leq \epsilon n$, then $\left(
7tk^4/n \right)^t \leq (7 \epsilon k^4)^t \leq (7/k)^t \leq
(7/k)^{(\log n)^2} = o(1/n).$

These estimates and (1) imply that the probability that $\chi(D[S]) \leq 2$ is
$o(1)$. This completes the proof.
\end{proof}

The $2$-colorability in the previous theorem cannot be decreased to $1$
due to the following theorem.

\begin{theorem}
If $D$ is a digraph with $\chi(D) \geq 3$ and of order $n$, then
it contains a directed cycle of length $o(n)$. 
\end{theorem}

\begin{proof}
In the proof we shall use the following digraph analogue of
Erd\H{o}s-Posa Theorem. Reed et al. \cite{RRST1996} proved that
for every integer $t$, there exists an integer $f(t)$ so that
every digraph either has $t$ vertex-disjoint directed cycles or a
set of at most $f(t)$ vertices whose removal makes the digraph
acyclic. Define $h(n) = \max\{t: tf(t) \leq n\}$. It is clear that $h(n) =
\omega(1)$. 

Let $c$ be the length of a shortest directed cycle in
$D$ and let $t:=h(n)$. If $D$ has $t$ vertex-disjoint directed cycles, then $ct
\leq n$ which implies that $c \leq \tfrac{n}{h(n)} = o(n)$.
Otherwise, there exists a set $S$ of
vertices with $|S| \leq f(t)$ such that $V(D) \backslash S$ is
acyclic. Since $\chi(D) \geq 3$, we have that $\chi(D[S]) \geq 2$,
which implies that $S$ contains a directed cycle of length at most
$|S| \leq f(t) \leq \tfrac{n}{t} = \tfrac{n}{h(n)}=o(n)$.
\end{proof}

\end{document}